\documentclass[12pt]{amsart}
\usepackage{enumerate}
\usepackage{hyperref}
 \theoremstyle{plain}
 \newtheorem{theorem}{Theorem}

 \newtheorem{lemma}{Lemma}
 \newtheorem{corollary}{Corollary}
\newtheorem{proposition}{Proposition}

\theoremstyle{definition}
\newtheorem{definition}{Definition}

\newtheorem{assumption}{Assumption}
 
 \newtheorem{remark}{Remark}

\newcommand{\A}{\mathbb A}

 \newcommand{\N}{\mathbb N}

 \newcommand{\R}{\mathbb R}

 \newcommand{\cC}{\mathcal C}
 \newcommand{\cD}{\mathcal D}
 
 \newcommand{\cW}{W^{1,1}}
 \newcommand{\cWD}{\cW_\cD}
 
 \newcommand{\cF}{\mathcal F}
\newcommand{\cG}{\mathcal G}

\newcommand{\ovu}{\overline{u}}
\newcommand{\unu}{\underline{u}}

\newcommand{\bh}{\overline{h}}

\newcommand{\bfh}{\mathbf{h}}

\newcommand{\ub}{\bar{u}}

 \newcommand{\af}{\alpha}
 \newcommand{\fui}{\varphi}
 \newcommand{\ep}{\varepsilon}
\newcommand{\be}{\beta}
\newcommand{\ga}{\gamma}

 \newcommand{\ka}{\kappa}
 
 \newcommand{\de}{\delta}
 
 \newcommand{\om}{\omega}
 
  \newcommand{\la}{\lambda}

 \newcommand{\dga}{\dot\gamma}

 \newcommand{\rip}{\rangle}
 \newcommand{\lip}{\langle}

\newcommand{\Lip}{\operatorname{Lip}}

\newcommand{\entre}{\setminus}

\begin{document}
\title [subriemannian Lagrangians depending on the unknown]
{Weak KAM theorems  for subriemannian Lagrangians
depending on the unknown function}
\author [R.  Iturriaga,]{Renato Iturriaga}
\address{Cimat, Valenciana Guanajuato 36000, M\'exico}
\author [H. S\'anchez Morgado]{ H\'ector S\'anchez Morgado}
\address{Instituto de Matem\'aticas, UNAM. Ciudad Universitaria C. P. 04510, Cd. de M\'exico, M\'exico.}
\email{hector@math.unam.mx}
\begin{abstract}
We extend some results of weak KAM theory to Lagrangians that are defined
  only on the horizontal distribution of a subriemannian
  manifold and depend on the unknown function.
  \end{abstract}
\maketitle

\section{Statement of results}
\label{sec:results}
Let $(M, \cD, \lip,\rip)$ be a subriemannian manifold such that
$\cD$ is bracket generating $TM$ and denote by $\pi:TM\to M$
and $\pi^*:T^*M\to M$ the natural projections.
\begin{definition}We say that
  \begin{itemize}
  \item 
 A continuous curve $\ga : [a, b]\to M$ is {\em absolutely continuous} iff
$\fui\circ\ga$ is  absolutely continuous for any smooth and compactly
supported $\fui:M\to\R$.
 \item 
An absolutely continuous curve $\ga:[a,b]\to M$ is {\em horizontal} if
$\dga(t)\in\cD$  for a. e. $t\in[a,b]$.
\item
 We denote by $\cWD([a,b])$
the set of horizontal absolutely continuous curves defined on the
interval $[a,b]$. 
\end{itemize}
For $\ga\in\cWD([a,b])$ we define its
subriemannian length by
\[\ell(\ga)=\int_a^b\|\dga(t))\|\ dt.\]
From that one obtains the {\em Carnot-Catatheodory} distance $d$ on $M$ which
defines a topology on $M$ that coincides with its
original manifold topology.
\end{definition}
We assume $M$ is compact and $L\in C^2(\cD\times\R)$ is
\begin{description}
  \item [Uniformly superlinear] For all $k\ge 0$ there is $C(k)\in\R$.
  such that \[L(v,0)\ge k\|v\|+C(k)\hbox{ for all }v\in\cD.\]
\item[Bounded monotone] There is $\la>0$ such that for all $v\in\cD$, $u\in\R$.
  \[-\la \le\partial_uL(v,u)<0.\]
\item [Strictly convex] There is  $\ka>0$ such that for all
  $u\in\R$, $v,w\in\cD$
  \[\partial_{vv}^2L(v,u)  (w,w)\ge\ka\|w\|^2.\]  
\end{description}
The subriemannian contact Hamiltonian $H:T^*M\times\R\to\R$ is defined by
\[H(p,u)=\max\{p(v)-L(v,u):v\in\cD, \pi(v)=\pi^*(p)\}\]
\begin{theorem}\label{main-evolution}
Given $\fui\in C(M)$ there exists a unique $u\in C([0,\infty)\times M)$ such that
\[u(t,x):
 =\inf\{\fui(\ga(0))+\int_0^tL(\dga(s),u(s,\ga(s)))\ ds:\ga\in\cWD([0,t]),\ \ga(t)=x\},\]

and so $u$ is a viscosity solution of the Cauchy problem
\begin{equation}\label{cauchy}
  \begin{cases} u_t+H(x,D_xu,u)=0\\ u(0,x)=\fui(x).   \end{cases}
\end{equation}
Define the Lax semigroup $T_t:C(M)\to C(M)$, $t\ge 0$ by $T_t[\fui](x)=u(t,x)$.
\end{theorem}
\begin{theorem}\label{main-stationary}
  Let $c_0$ be the Ma\~n\'e critical value of $L(\cdot,0)$ and denote by $T_t^*$ the Lax semigroup  of $L+c_0$, then 
  \begin{enumerate}[(a)]
  \item 
If $\fui\in\Lip(M,d)$, the family $T^*_t[\fui]$ is uniformly bounded and equi-Lipschitz.
\item There is $\psi\in\Lip(M,d)$ such that $T^*_t[\psi]=\psi$ for
  any $t\ge 0$, and so  $\psi$ is a viscosity solution of the Hamilton-Jacobi equation 
\begin{equation}\label{eq:HJ}
  H(x,D_xu,u)=c_0
\end{equation}
\item Under assumption \ref{Bar} below, if $\fui\in\Lip(M,d)$, the uniform limit
  $u=\lim\limits_{t\to\infty}T^*_t[\fui]$ exists and it is the unique
  solution of \eqref{eq:HJ}.
  \end{enumerate}
\end{theorem}
A subriemannian structure on $M$ defines a symmetric bundle map
 $\be:T^*M\to TM$,   called a {\em cometric,} by the conditions
  \begin{enumerate}
  \item $\be(T_x^*M)=\cD_x,$
  \item $\lip\be(p),v\rip=p(v)$,
  \end{enumerate}
and conversely, a symmetric  bundle map $\be:T^*M\to TM$ of constant rank
defines a subriemannian structure, see \cite{M}.
Defining $L^*:\cD\to\R$ by
\[L^*(w,u)=\max\{\lip w,v\rip-L(v,u):v\in\cD , \pi(v)=\pi(w)\},\]
we have $H(p)=L^*(\be(p))$.
Thus, subriemannian Hamiltonians are pseudo-coercive, and the uniqueness
of the solution of \eqref{eq:HJ} is guaranteed under the assumption in page 35
of the book \cite{Ba}
\begin{assumption}\label{Bar}
  For each $R>0$ there is $\Phi_R:\R\to\R$ such that
for all $(x,p)\in T^*M$, $|u|\le R$
\begin{equation}
  \label{eq:bar}
 \Big|\frac{\partial H}{\partial x}(x,p,u)\Big|\le (1+|p|) \Phi_R(H(x,p,u)).
\end{equation}
\end{assumption}
This equation makes perfect sense if $T^*M$ is trivial. Embedding $M$ in an
euclidean space we can make sense of \ref{eq:bar}.

The Lagrangian $L(x,v,u)=\frac 12 \|v\|^2-U(x,u)$ with $0<\partial_uU\le\la$,
has  Hamiltonian $H(x,p,u)=\frac 12 \|\be(p)\|^2+U(x,u)$ which
satisfies \eqref{eq:bar} with 
\[\Phi_R(z)=\max_{M\times[-R,R]}(|\partial_xU|,\|D_x\be\|)(1+\sqrt{(z-\min U(\cdot,R))_+}).\]
 \section{Tonelli's Theorem for non-autonomous Lagrangians}
\label{sec:tonelli}
\begin{definition}\quad
  \begin{itemize} 
   \item  A family $\cF$  of curves $\ga:[a,b]\to M$ is absolutely
  equicontinuous if $\forall\ep>0$, $\exists\de>0$
  s.t. $\forall\ga\in\cF$ and any  disjoint subintervals $ ]a_1,b_1[,\cdots,]a_N,b_N[$ 
 of $[a,b]$:
  \[\sum_{i=1}^Nb_i-a_i<\de\implies \sum_{i=1}^Nd(\ga(b_i),\ga(a_i))<\ep\]
\item A
 family $\cG\subset L^1([a,b],m)$,  $m$ the Lebesgue measure on
 $[a,b]$, is uniformly integrable 
if given $\ep>0$ there is $\de>0$ such that 
\[f\in\cG, E\subset[a,b]\hbox{ measurable}, m(E)<\de\implies
  \int_E|f|<\ep.\]
   \end{itemize}
 
\end{definition}
\begin{remark}\label{basic}\quad
  \begin{enumerate}
  \item  A family $\cF\subset\cWD([a,b])$ such that $\{\|\dga\| : \ga\in\cF\}$
    is uniformly integrable, is absolutely equicontinuous.
  \item An absolutely equicontinuous family is equicontinuous.
  \item A uniform limit of absolutely equicontinuous  functions
    is absolutely continuous.
    \end{enumerate}
\end{remark}
\begin{theorem}{\cite{SM}}\label{hor}
  Let $(\ga_n)_n$ be a sequence in $\cWD([a,b])$ that converges uniformly to
  $\ga$, and such that
  $( \|\dga _n\|)_n$ is uniformly integrable.   Then $\ga\in\cWD([a,b])$.
  \end{theorem}

Let $L\in C(I\times\cD)$ be such that for each
$t\in I$, $x\in M$, the map $\cD_x\to\R$, $v\mapsto L(t,v)$ is twice
differentiable with $\partial^2L:I\times\cD\to\cD^*\otimes\cD^*$ 
continuous and the following properties hold
\begin{description}
\item [Uniform superlinearity] For all $k\ge 0$ there is $C(k)\in\R$.
  such that \[L(t,v)\ge k\|v\|+C(k)\hbox{ for all }(t,v)\in I\times\cD.\]
\item [Uniform boundedness] For all $r\ge 0$, we have
  \[A(r)=\sup\{L(t,v):\|v\|\le r\}<+\infty.\]
\item [Strict convexity] There is  $\ka>0$ such that for all  $t\in I$, $v,w\in\cD$ 
 \[\partial^2L (t,v)(w,w)\ge\ka\|w\|^2.\]  
\end{description}
For $\ga\in\cWD([a,b])$ we define the action
\[A_L(\ga)=\int_a^bL(t,\dga(t))\ dt.\]
\begin{lemma}\label{UI}
  For $c\in\R$ let
  $\cF(c)=\{\ga\in\cWD([a,b]): A_L(\ga)\le c\}.$
  Then the family
  $\{\|\dga(t)\|:\ga\in\cF(c)\}$ is uniformly integrable.
\end{lemma}
\begin{proof}
 By the uniform superlinearity,
      \[L(t,v)\ge K\|v\|+C(K).\]
      Given $\ep>0$ take $K>0$ such that 
  $$c-C(0)(b-a)<K\ep.$$

  Let $\ga\in\cF(c)$ and $E\subset[a,b]$ be measurable, then  
\[C(K)m(E)+K\int_E\|\dga(t)\|\ dt
   \le\int_E L(t,\dga(t))\ dt\]
and
   \[C(0)m([a,b]\entre E)\le \int_{[a,b]\entre E} L(t,\dga(t))\ dt \]
  Adding the inequalities we get 
  \[(C(K)-C(0))m(E)+C(0)(b-a)+K\int_E
    \|\dga(t)\|\ dt    \le A_L(\ga)\le c\]
which gives
\[\int_E\|\dga(t)\|\ dt\le\frac{c-C(0)(b-a)}K+
  \frac{C(0)-C(K)}Km(E) <\ep+\frac{C(0)-C(K)}{K}m(E).\]
This gives the uniform integrability of $\{\|\dga(t)\|:\ga\in\cF(c)\}$.
\end{proof}
  
\begin{theorem}\label{Cle}
   Let $(\ga_n)_n$ be a sequence in $\cWD([a,b])$ converging uniformly to
   $\ga$ such that
   \[\liminf_{n\to\infty}A_L(\ga_n)<+\infty.\]
   Then $\ga\in\cWD([a,b])$ and
\[A_L(\ga)\le \liminf_{n\to\infty}A_L(\ga_n).\]
\end{theorem}
\begin{proof}
  
Let $l=\liminf\limits_{n\to\infty}A_L(\ga_n)$,  extracting a subsequence
still denoted $\ga_n$, we have  $l=\lim\limits_{n\to\infty}A_L(\ga_n)$, and forgetting some of the first curves $\ga_n$, 
we can suppose that $\ga_n\in\cF(l+1)$ for all $n$. 
Lemma \ref{UI} implies that $(\|\dga_n\|)_n$ 
is uniformly integrable, and Theorem \ref{hor} that $\ga\in\cWD([a,b])$.
  
Let us show how we can reduce the proof to the case where 
$M$ is an open subset of $\R^d$,  $d =\dim M$ and the horizontal 
distribution is trivial. The lagrangian $L$ is bounded
below by $C(0)$ . If $[a',b']\subset[a,b]$, for all $n$ we have

\[A_L(\ga_n|[a',b'])\le
  A_L(\ga_n)-C(0)(b-b'+a'-a)\]
so that
\[\liminf_{n\to\infty}A_L(\ga_n|[a',b'])<+\infty.\]
Let now the partition $a_0 = 0 < a_1 < \ldots < a_p = 1$ and
$U_1, \ldots, U_p$ be such that
the horizontal distribution is trivial on $U_i$  and
$\ga([a_{i-1},a_i])\subset U_i$, $i=1,\ldots,p$,
It is enough to prove that  
\[A_L(\ga|[a_{i-1},a_i])\le \liminf_{n\to\infty}A_L(\ga_n|[a_{i-1},a_i])\]
because that implies and 
\begin{align*}
  A_L(\ga) =\sum_{i=1}^pA_L(\ga|[a_{i-1},a_i]) &\le
\sum_{i=1}^p\liminf _{n\to\infty} A_L(\ga_n|[a_{i-1},a_i])\\
&\le \liminf_{n\to\infty}\sum_{i=1}^pA_L(\ga_n|[a_{i-1},a_i])=\liminf_{n\to\infty}A_L(\ga_n)
\end{align*}
In the sequel we suppose that $M=U$ is an open set of $\R^d$, the
horizontal distribution is $U\times\R^m$ with $\langle ,\rangle$ the
euclidean inner product in $\R^m$, and that $\ga([a,b])$ and 
all $\ga_n([a,b])$ are contained in $U$.
We will write $\dga(t)=(\ga(t), \zeta(t))$, $\dga_n(t)=(\ga_n(t), \zeta_n(t))$.

   \begin{lemma}\label{LemCle}
     Let $U\subset\R^d$ be open, $L\in C([a,b]\times U\times\R^m)$ be
     twice differentiable in the the third variable with
     $\partial_{vv}L\in C([a,b]\times U\times\R^m,\R^{m^2})$, stricly convex
       and uniformly superlinear in the third variable and with
$\{L(t,x,v):\|v\|\le c\}$ bounded for each $c>0$.
     Given $K\subset U$ compact, $C>0$ and $\ep>0$, there exists $\de>0$
  such that if $x\in K$, $| x-y|\le\de$, $\|v\|\le C$ and $w\in\R^m$,
  then
  \[  L(t,y,w)\ge L(t,x,v)+\partial_v L(t,x,v)\cdot(w-v) -\ep.\]
  \end{lemma}
  For a fixed constant $C$ let
   \[E_C=\{t\in[a,b]:\|\zeta(t)\|\le C\}  \]
  Given $\ep> 0$, the compact set $K=\ga([a, b])\cup\bigcup_{n\in\N} \ga_n[a, b]$ 
and the  constant $C$ fixed above, we apply Lemma \ref{LemCle} to get $\de >0$
  satisfying its conclusion. By the compactness of $[a,b]$ and the
  continuity of $\ga$, 
  there is $\eta>0$ such that  $t\in[a,b]$, $d(x,\ga(t)) < \eta$
  imply $|x-\ga(t)| < \de.$
 Since  $\ga_n$ converges uniformly to $\ga$,
  there exists an integer $n_0$ such that, for each $n\ge n_0$
  we have $d(\ga_n(t),\ga(t)) < \eta$ for each $t\in[a, b]$.
Hence, for each $n\ge n_0$ and almost all $t \in E_C$, we have
  \[L (\ga_n(t),\zeta_n(t))\ge L (\ga(t), \zeta(t)) +
    \partial_v L(t,\ga(t), \zeta(t))\cdot(\zeta_n(t) -\zeta(t))-\ep,\]
 and from the uniform superlinearity, $L(t,\ga_n(t), \zeta_n(t))\ge C(0)$
 a. e.. Thus
  \begin{align}\nonumber
    A_L(\ga_n)&\ge \int_{E_C}L(t,\ga(t), \zeta(t))\ dt
                       +C(0)m([a,b]\entre E_C)\\
& +\int_{E_C}\partial_vL(t,\ga(t), \zeta(t))\cdot (\zeta_n(t) -\zeta(t))\ dt
\-\ep m(E_C) \label{1}
  \end{align}
  Since $\{\|\zeta_n(t)\|\}$ is uniformly integrable, $\zeta_n(t)$ converges to
  $\zeta(t)$ in the weak topology $\sigma(L^1,L^\infty)$.
Since $\|\zeta(t)\|\le C$ for $t\in E_C$, the function $\chi_{E_C}(t)\partial_vL(t,\ga(t), \zeta(t))$
is bounded. Thus
\[\int_{E_C}\partial_v L(t,\ga(t), \zeta(t))\cdot (\zeta_n(t) -\zeta(t))\to 0,\]
as $n\to\infty$. Taking limit in \eqref{1} we have
\[l=\lim_{n\to\infty}A_L(\ga_n)\ge \int_{E_C}L(t,\ga(t), \zeta(t))\ dt
  +C(0) m([a,b]\entre E_C)-\ep m(E_C).\]
Letting $\ep\to 0$ we have
\begin{equation}
  \label{eq1}
  l=\lim_{n\to\infty}A_L(\ga_n)\ge \int_{E_C}L(t,\ga(t), \zeta(t))\ dt
  +C(0) m([a,b]\entre E_C).
\end{equation}
Since $\zeta(t)$ is defined and finite for almost all $t\in[a,b]$ we
have that $E_C\nearrow E_\infty$ as $C\nearrow +\infty$ with
$m([a,b]\entre E_\infty)=0$. Since $L(t,\ga(t), \zeta(t))$ is bounded below by $C(0)$,
the monotone convergence theorem gives
\[\int_{E_C}L(t,\ga(t), \zeta(t))\ dt\to\int_a^bL(t,\ga(t), \zeta(t))\ dt \hbox{ as }C\to+\infty.\]
Letting $C\nearrow+\infty$ in \eqref{eq1} we finally obtain
\[l=\lim_{n\to\infty}A_L(\ga_n) \ge A_L(\ga)\]
\end{proof}
\begin{corollary}\label{low-semi}
  The action
  $A_L:\cWD([a,b])\to \R\cup\{+\infty\}$ is lower semicontinuous
  for the topology of uniform convergence in $\cWD([a,b])$.
\end{corollary}
\begin{proof}
  Let $\ga_n$ be a sequence in $\cWD([a,b])$ that converges
  uniformly to $\ga\in \cWD([a,b])$. We must show that
  \[\liminf_{n\to\infty}A_L(\ga_n)\ge A_L(\ga).\]
If $\liminf\limits_{n\to\infty}A_L(\ga_n)=+\infty$ there is nothing to prove.
In other case, the result follows from Theorem \ref{Cle}
\end{proof}

\begin{corollary}[Tonelli's Theorem]\label{tonelli}
  If $M$ is complete and $K\subset M$ is compact, $c\in\R$
  then the set 
\[\cF(K,c)=\{\ga\in\cWD([a,b]): \ga([a,b])\cap K\ne\emptyset,
A_L(\ga)\le c\}\]
is a compact subset of $\cWD([a,b])$ for the topology of
uniform convergence.
\end{corollary}
\begin{proof}
By the compactness of $K$ and Theorem \ref{Cle}, the subset
  $\cF(K,c)$ is closed in the space of continuous curves $C([a,b],M)$.
  By the uniform superlinearity,
  $L(t,v)\ge\|v\|+C(1)$ for all $t\in I$, $v\in\cD$. Thus
  for any $\ga\in\cWD([a,b])$ and any $t,s\in[a,b]$ with $t\le s$, we have 
 \[C(1)(s-t)+\int_t^s\|\dga\|\le A_L(\ga)\]
If $\ga\in\cF(c,K)$ we have
\[d(\ga(s),\ga(t))\le c-C(1)(b-a) ,\]
so
\[\ga([a,b])\subset\{y\in M:d(y,K)\le c-C(1)(b-a) \}.\]
Since $\cF(c)$ is absolutely equicontinuous by  Theorem \ref{UI} and
Remark \ref{basic} (1),
the Arzela-Ascoli Theorem implies that $\cF(K,c)$ is a compact subset
of $\cWD([a,b])$ for the topology of uniform convergence.
\end{proof}
\begin{definition} For $x,y\in M$, $a,b\in\R$ $a<b$
let $\cC_{a,b}(x,y):=\{\ga\in\cWD([a,b]): \ga(a)=x, \ga(b)=y\}$
and define the  {\em minimal action} by
  \[\bh(x,y;a,b):=h_{a,b}(x,y):=\inf\{A_L(\ga):\ga\in\cC_{a,b}(x,y)\}\]
\end{definition}
We denote $h_t(x,y):=h_{0,t}(x,y)$
\begin{corollary}[Tonelli minimizers]\label{ton-min}
    If $M$ is complete, for each $x,y\in M$ 
 and each $a,b\in\R$, $a<b$, there exists $\ga\in\cC_{a,b}(x,y)$,
 $A_L(\ga) =h_{a,b}(x,y)$.
\end{corollary}
\begin{proof}
  Set $\bar C=\inf\{A_L(\ga):\ga\in\cC_{a,b}(x,y)\}$.
  By Corollary \ref{tonelli}, the set 
\[\{\ga\in\cC_{a,b}(x,y):A_L(\ga)\le\bar C+1\}\]
is a compact subset of $\cWD([a,b])$ for the topology of
uniform convergence.
Choose $\ga_n\in\cC_{a,b}(x,y)$ such that $A_L(\ga_n)<\bar C+\dfrac 1n$.
Then $\ga_n$ has a subsequence that converges uniformly
to some $\ga\in \cC_{a,b}(x,y)$. By Theorem \ref{Cle}, $A_L(\ga)=\bar C$. 
\end{proof}
  \section{Proof of Theorem \ref{main-evolution}}
  \label{sec:contact-lagrangians}
 We extend ideas in \cite{SWY} to our situation.
  
Let $M$ and $L\in C^2(\cD\times\R)$ be as in subsection
\ref{sec:results}.

Let $A(r)=\sup\{L(v,0):\|v\|\le r\}$.

For $\fui\in C(M)$ and $u\in C([0,T]\times M)$ we define
$\A_\fui[u]:[0,T]\times M\to\R$ by
  \[\A_\fui[u](t,x):
    =\inf\{\fui(\ga(0))+\int_0^tL(\dga(s),u(s,\ga(s)))\ ds:\ga\in\cWD([0,t]),\ \ga(t)=x\}.\]
  \begin{remark}
      As in the proof of Corollary \ref{ton-min}, it follows from Corollary \ref{tonelli},
  Theorem \ref {Cle}, and the continuity of $\fui$ that the infimum is actually a minimum.
\end{remark}
  We omit the subscript $\fui$ when that does not cause confusion.
  
  Defining  $L^u:[0,T]\times\cD\to\R$ by $L^u(t,v)=L(v,u(t,x))$ we have 
  \[A^u(r)=A(r)+\la\max u_-,\quad C^u(r)=C(r)-\la\max u_+\]
  \[\A[u](t,x)=\inf_{y\in M}\fui(y)+h^u_t(y,x)\]
where $h_t^u$ is the minimal action for $L^u$.

\begin{proposition}
  Let $\fui\in C(M)$, then $\A_\fui(C([0,T]\times M)\subset C([0,T]\times M)$ 
\end{proposition}

\begin{proof} Let $u\in C([0,T]\times M)$, we have
  \[\min\fui+tC^u(0)\le\A[u](t,x)\le \max\fui+tA^uu(0).\]
  
Let $s>t$, $x\in M$. For any $y\in M$
\[\A[u](s,x)\le\fui(y)+h^u_s(y,x)\le \fui(y)+h^u_t(y,x)+A^u(0)(s-t)\]
thus
\begin{equation}
  \label{eq:s,t}
  \A[u](s,x)\le \A[u](t,x) +A^u(0)(s-t).
  \end{equation}
Joining $x$ and $z$ with an a.e. unit-speed minimizing geodesic, for
all $y\in M$ we have
\begin{align}\nonumber
  h^u_{t+d(x,z)}(y,z)&\le h^u_t(y,x)+A^u(1)d(x,z)\hbox{ and so}\\
  \A[u](t+d(x,z),z)&\le\A[u](t,x)+A^u(1)d(x,z). \label{t+d}
\end{align}                     
Since $u$ is uniformly continuous it has a continuity modulus
$\om:[0,\infty[\to[0,\infty[$

Suppose that $\fui\in\Lip(M,d)$.
For $x\in M$ and $s>t$ let $r=s-t$. For $\ga\in\cWD([0,s])$ with
$\ga(s)=x$ define $\be:[0,t]\to M$
by $\be(\tau)=\ga(\tau+r)$, so that $\be(0)=\ga(r)$, $\be(t)=x$.
\begin{align*}
  \fui(\ga(0))+\int_0^sL^u(\tau,\dga(\tau)) d\tau
=\fui(\ga(0))+\int_0^rL^u(\tau,\dga(\tau)) d\tau
+\int_0^t L^u(\tau+r,\dot\be(\tau)) d\tau\\
 \ge\fui(\ga(0))+\Lip(\fui)\int_0^{r}|\dga|+C^u(\Lip(\fui)r
 +\int_0^t( L^u(\tau,\dot\be(\tau))-\la\om(r)) d\tau\\
\ge\fui(\ga(0))+\Lip(\fui)d(\be(0),\ga(0))+\int_0^t L^u(\tau,\dot\be(\tau)) d\tau+
C^u(\Lip(\fui))r-\la t\om(r)\\
\ge\fui(\be(0)) +\int_0^t L^u(\tau,\dot\be(\tau)) d\tau+
C^u(\Lip(\fui))r-\la t\om(r). 
\end{align*}
Thus
\[\A[u](s,x)\ge\A[u](t,x)+C^u(\Lip(\fui))(s-t)-\la t\om(s-t),\]
in particular
\begin{equation}
  \label{eq:t,s}
  \A[u](s,x)\ge \fui(x)+C^u(\Lip(\fui))s.
  \end{equation}
Let $x,y\in M$
\begin{align*}
\A[u](t,y)&\le\A[u](t+d(x,y),y)-C^u(\Lip(\fui))d(x,y)+\la t\om(d(x,y))\\  
&\le\A[u](t,x)+(A^u(1)-C^u(\Lip(\fui)))d(x,y) +\la t\om(d(x,y))
\end{align*}
When $\Lip_1(u):=\sup\limits_{0\le t<s\le T}\dfrac{\|u(s,\cdot)-u(t,\cdot)\|_\infty}{s-t}<+\infty$
we have that
\begin{align*}
  \A[u](s,x)\ge\A[u](t,x)&+[C^u(\Lip(\fui))-\la t\Lip_1(u)](s-t),\quad s>t\\
|\A[u](t,y)-\A[u](t,x)|&\le [A^u(1)-C^u(\Lip(\fui))+\la t\Lip_1(u)]d(x,y).
\end{align*}
If $\fui$ is only continuous,
given $\ep>0$ there is $\psi$ Lipschitz such that $\|\fui-\psi\|_\infty<\ep$.
For $x\in M$ let $\ga\in\cWD([0,t])$ with $\ga(t)=x$,
\begin{multline*}
  \psi(\ga(0))+\ep+\int_0^tL^u(\tau, \dga(\tau)) d\tau
\ge \fui(\ga(0))+\int_0^tL^u(\tau, \dga(\tau)) d\tau\\
  \ge\psi(\ga(0))-\ep+\int_0^tL^u(\tau, \dga(\tau)) d\tau.
\end{multline*}
  Thus $\A_\psi[u]+\ep\ge\A_\fui[u]\ge\A_\psi[u]-\ep$ and so
\begin{align*}
  \|\A_\fui[u](t,\cdot)-\A_\fui[u](s,\cdot)\|_\infty
  &\le\|\A_\psi[u](t,\cdot)-\A_\psi[u](s,\cdot)\|_\infty+2\ep\\
&\le -C^u(\Lip(\psi))(s-t)+\la t\om(s-t)+2\ep\\
  |\A_\fui[u](t,x)-\A_\fui[u](t,y)|&\le |\A_\psi[u](t,x)-\A_\psi[u](t,y)|+2\ep\\
&\le(A^u(1)-C^u(\Lip(\psi)))d(x,y) +\la t\om(d(x,y))+2\ep
\end{align*}
\end{proof}

\begin{lemma}\label{A-fix}
 Let $\fui\in C(M)$, then $\A_\fui$ has a unique fixed point.
\end{lemma}
\begin{proof}
  Let $u,\ub\in C([0,T]\times M)$.
  Let $\ga\in\cWD([0,t])$ with $\ga(t)=x$ be such that
  \[\A[u](t,x)=\fui(\ga(0))+\int_0^tL(\dga(\tau),u(\tau,\ga(\tau)))\ d\tau.\]
  Then
  \begin{align*}
     \A[\ub](x,t)-\A[u](x,t)&\le \int_0^tL(\dga(\tau),\ub(\tau,\ga(\tau)))-L(\dga(\tau),u(\tau,\ga(\tau)))\ d\tau\\
&\le \la\int_0^t|\ub(\tau,\ga(\tau))-u(\tau,\ga(\tau))|\ d\tau\le \la t\|\ub-u\|_\infty 
  \end{align*}
  Similarly
 \[|\A^2[\ub](x,t)-\A^2[u](x,t)|\le  \la\int_0^t \la s\|\ub-u\|_\infty\ ds=\frac{(\la t)^2}2\|\ub-u\|_\infty\]
 and inductively
 \[|\A^n[\ub](x,t)-\A^n[u](x,t)|\le \frac{(\la t)^n}{n!}\|\ub-u\|_\infty.\]
 Thus
 \[|\A^{n+1}[u](x,t)-\A^n[u](x,t)|\le \frac{(\la t)^n}{n!}\|\A[u]-u\|_\infty\]
 which implies that $\A^n[u]$ is a Cauchy sequence for the uniform
 norm and so it converges to a fixed point $w$ of $\A$. If $\bar w$ is
 another fixed point
 \[\|\bar w-w\|_\infty\le \frac{(\la t)^n}{n!}\|\bar w-w\|_\infty\]
and then $\bar w=w$. 
\end{proof}
\begin{proposition}\label{fix=sol}
Let $\fui\in C(M)$ and let $u\in C([0,T]\times M)$ be the fixed point of
$\A_\fui$, $0\le s_1<s_2\le T$, $x\in M$
\begin{enumerate}[(a)]
\item For any $\af\in\cWD([s_1,s_2])$ we have
  \[u(s_2, \af(s_2))-u(s_1,\af(s_1))\le
    \int_{s_1}^{s_2}L(\dot\af(\tau),u(\tau,\af(\tau)))\ d\tau.\]
\item For $\ga\in\cWD([0,t])$ with $\ga(t)=x$ such that
  \[u(t,x)=\fui(\ga(0))+\int_0^tL(\dga(\tau),u(\tau,\ga(\tau)))\ d\tau,\]
  and any $s\in[0,t]$ we have that
  \[u(t,x)=u(s,\ga(s))+\int_s^tL(\dga(\tau),u(\tau,\ga(\tau)))\ d\tau.\]
\item $u$ is a viscosity solution of the Cauchy problem \eqref{cauchy}.
  \end{enumerate}
\end{proposition}
\begin{proof}
  {\it (a)} Take $\ga\in\cWD([0,s_1])$ with $\ga(s_1)=\af(s_1)$ such that
  \[u(s_1,\ga(s_1))=\fui(\ga(0))+\int_0^{s_1}L(\dga(\tau),u(s,\ga(\tau)))\ d\tau.\]
\[u(s_2,x)\le \fui(\ga(0))+\int_0^{s_1}L(\dga(\tau),u(s,\ga(\tau)))\ d\tau+
\int_{s_1}^{s_2}L(\dot\af(\tau),u(\tau,\af(\tau)))\ d\tau\]
  
  \begin{align*}(b)\qquad
    u(t,\af(s_2))&=\fui(\ga(0))+\int_0^sL(\dga(\tau),u(\tau,\ga(\tau)))\ d\tau
  +\int_s^tL(\dga(\tau),u(\tau,\ga(\tau)))\ d\tau\\
  &\ge u(s,\ga(s))+\int_s^tL(\dga(\tau),u(\tau,\ga(\tau)))\ d\tau\ge u(t,x).
    \end{align*}
{\it (c)} By standard arguments, {\it (a)} implies that $u$ is a
    viscosity subsolution of \eqref{cauchy} and {\it (b)} implies
    that $u$ is a viscosity supersolution of \eqref{cauchy}.
  \end{proof}

Define the Lax semigroup $T_t:C(M)\to C(M)$, $t\ge 0$ by $T_t[\fui](x)=u(t,x)$,
where $u|_{[0,T]\times M}$ is the unique fixed point of
$\A_\fui :C([0,T]\times M)\to C([0,T]\times M)$.
\begin{corollary}\label{weakKAM}
  Let $\fui\in C(M)$, if $T_t[\fui]=\fui$ then
  \begin{enumerate}[(a)]
  \item For any $\af\in\cWD([s_1,s_2])$ we have
  \[\fui(\af(s_2))-\fui(\af(s_1))\le
    \int_{s_1}^{s_2}L(\dot\af(\tau),\fui(\af(\tau)))\ d\tau.\]
\item  For any $x\in M$, there is $\ga_x\in\cWD((-\infty, 0])$ with
  $\ga_x(0)=x$ such that for any $t<0$
  \[\fui(x) =\fui(\ga_x(t))+\int_t^0L(\dga_x(s),\fui(\ga_x(s)))\ ds\]
  \item $\fui$ is a viscosity solution of
\[H(x,D_xu,u)=0.\]
  \end{enumerate}
  \end{corollary}
  \begin{proof}
    {\it (a)} follows from item {\it (a)} of Proposition \ref{fix=sol}.

    {\it (b)}. Let $x\in M$, from item {\it (b)} of Proposition \ref{fix=sol},
  for any $T>0$ there is $\ga^T\in\cWD([-T, 0])$ with
  $\ga^T(0)=x$ such that for $-T\le t<0$
  \[\fui(x) =\fui(\ga^T(t))+\int_t^0L(\dga^T(s),\fui(\ga^T(s)))\ ds\]
  Thus \[\int_t^0L(\dga^T(s),\fui(\ga^T(s)))\ ds\le 2\|\fui\|_\infty\]
By Tonelli's Theorem there is a sequence $T_j\to\infty$
such that $\ga^{T_j}|[t,0]$ converges uniformly for the metric
$d$. Aplying this argument to a sequence $t_k\to\infty$ and using a
diagonal trick one gets a sequence $\sigma_n\to\infty$ and a curve
$\ga_x\in\cWD(]-\infty,0])$ such that
$\ga^{\sigma_n}|[t,0]$ converges $d$-uniformly 
to $\ga_x|[t,0]$  for any $t<0$. By the continuity of $\fui$ and
Proposition \ref{Cle},
\begin{align*}
  \fui(x) &=\fui(\ga_x(t))+\int_t^0L(\dga_x(s),\fui(\ga_x(s)))\ ds\\
          &\le\liminf_{n\to\infty}\fui(\ga^{\sigma_n}(t))+\
            \int_t^0L(\dga^{\sigma_n}( s),\fui(\ga^{\sigma_n} (s)))\ ds=\fui(x)
\end{align*}
{\it (c)} follows from {\it (a)} and {\it (b)} by standard arguments.
\end{proof}
  \begin{corollary}  For $s,t\ge 0$, $ T_s\circ T_t=T_{s+t}$.
  \end{corollary}
  \begin{proof} Let $\fui\in C(M)$ and 
    $v(\tau,x)=u(\tau+t,x)=T_{\tau+t}[\fui](x).$
      Let $\ga\in\cWD([0,s])$ with $\ga(s)=x$ and define $\af:[t,s+t]\to M$
    by $\af(\tau)=\ga(\tau-t)$
    \begin{align*}
    v(s,x)=u(s+t,x)&\le
      u(t,\af(t))+\int_t^{s+t}L(\dot\af(\tau),u(\tau,\af(\tau)))\ d\tau\\
     &= T_t[\fui](\ga(0)) +\int_0^sL(\dga(\tau),v(\tau,\ga(\tau)))\ d\tau
    \end{align*}
      There is $\af:[t,s+t]\to M$ with $\af(s+t)=x$ such that
     \[u(s+t,x)= u(t,\af(t))+\int_t^{s+t}L(\dot\af(\tau),u(\tau,\af(\tau)))\  d\tau\]
    Defining $\be:[0,s]\to M$ by $\be(\tau)=\af(\tau+t)$ we have $\be(s)=x$ and
     \[v(s,x)= T_t[\fui](\be(0)) +\int_0^sL(\dot\be(\tau),v(\tau,\be(\tau)))\ d\tau.\]
     So 
     \[ v(s,x)=\inf\{T_t[\fui](\ga(0))
       +\int_0^sL(\dga(\tau),v(\tau,\ga(\tau)))\ d\tau:\ga\in\cWD([0,s),\ \ga(s)=x\}\]
and then $T_{s+t}[\fui](x)=v(s,x)=T_s[T_t[\fui]](x)$.
\end{proof}
\section{Proof of Theorem \ref{main-stationary}}

\begin{proposition}[Monotone]\label{monotony}
 Let $\fui, \psi\in C(M)$ be such that $\fui\le\psi$, and $t\geq 0$, then
 $T_t[\fui]\le T_t[\psi]$. 
\end{proposition}
\begin{proof}
We  assume by contradiction, that there exist
$t_1>0$ and $x_1\in M$ such that $T_{t_1}[\fui](x_1) > T_{t_1}[\psi](x_1)$.
Let $\ga\in\cWD([0,t_1])$ with $\ga(t_1)=x_1$ be such that
\[T_{t_1} [\psi](x_1)=\psi(\ga(0))+\int_0^{t_1}L(\dga(\tau),T_\tau
[\psi](\ga(\tau)))\ d\tau. \]
Define
\[F(t)=T_{t}[\fui](\ga(t)) - T_{t}  [\psi](\ga(t)).\]
We have that $F$ is continuous and $F(t_1)=T_{t_1}[\fui](x_1)- T_{t_1}[\psi](x_1)>0$. 

Since $F(0)=\fui(\ga(0)) - \psi(\ga(0))\le 0,$ there exists
$t_0\in [0,t_1)$ such that $F(t_0)=0$ and $F(t)\ge 0$ for any $t\in[t_0,t_1]$.

By  Proposition \ref{fix=sol}, for $t\in [t_0,t_1]$ we have
\[F(t_1)\le
 F(t_0)   +\int_{t_0}^{t_1} L(\dga(\tau),T_\tau
[\fui](\ga(\tau))) -L(\dga(\tau),T_\tau [\psi](\ga(\tau)))\ d\tau
\leq 0+0,\]
which is a contradiction.
\end{proof}
\begin{proposition}[Non-expansive]\label{non-exp}
  Let $\fui,\psi\in C(M)$, then
  \[\|T_t[\fui]-T_t[\psi]\|_\infty\leq \|\fui-\psi\|_\infty.\]
\end{proposition}
\begin{proof}
  Assume that there are
  $t_1>0$ and $x_1\in M$ such that
  $T_{t_1}[\fui](x_1)- T_{t_1}[\psi](x_1) >\|\fui-\psi\|_\infty$.
Let $\ga\in\cWD([0,t_1])$ with $\ga(t_1)=x_1$ be such that
\[T_{t_1} [\psi](x_1)=\psi(\ga(0))+\int_0^{t_1}L(\dga(\tau),T_\tau
[\psi](\ga(\tau)))\ d\tau. \]
Define as in Proposition \ref{monotony}
\[F(t)=T_{t}[\fui](\ga(t)) - T_{t}  [\psi](\ga(t)).\]
We have that $F$ is continuous and $F(t_1) >\|\fui-\psi\|_\infty$. 
Since $F(0)\le \|\fui-\psi\|_\infty,$ there exists
$t_0\in [0,t_1)$ such that $F(t_0)=\|\fui-\psi\|_\infty$ and
$F(t)\ge \|\fui-\psi\|_\infty\ge 0$ for any $t\in[t_0,t_1]$.
By  Proposition \ref{fix=sol}, for $t\in [t_0,t_1]$ we have
\[F(t_1)\le
    F(t_0)+\int_{t_0}^{t_1} L(\dga(\tau),T_\tau
[\fui](\ga(\tau))) -L(\dga(\tau),T_\tau [\psi](\ga(\tau)))\ d\tau
\leq\|\fui-\psi\|_\infty+0,\]
which is a contradiction. Therefore
$T_t[\fui](x)- T_t[\psi](x) \le\|\fui-\psi\|_\infty$ for any $t>0$ and
$x\in M$ and similarly $T_t[\psi](x)- T_t[\fui](x) \le\|\fui-\psi\|_\infty$.
\end{proof}
 Let $c[0]$ and  $h_t(x,y)$ be the Ma\~n\'e critical value and the
 minimal action of $L(\cdot,0)$. Denote $L_0=L+c[0]$ and 
 $T_t^*$ its Lax semigroup.

According to Corollary 4 in \cite{SM},
for each $\de>0$ there is $C_\de >0$ such that
\[|h_t(x,y)+c[0]t|\le C_\de\ \hbox{ for }\ x,y\in  M, t\ge\de.\]
\begin{lemma}[Uniform bound]\label{unif-bound}
For every $\fui\in C(M)$, 
$\sup\limits_{t\ge 0}\|T^*_t \fui\|_\infty<\infty$.
\end{lemma}
\begin{proof}
 Let $u(t,x):=T^*_t\fui(x)$ then $u$ is bounded on $[0,1]\times M$:

We first show that $u(t,x)$ is uniformly bounded from below.
Assume  there is $(t,x)\in M\times(1,+\infty)$ with $u(t,x) < 0$,
otherwise $u$ is bounded from below.
There exists $\ga\in\cWD([0,t])$ with $\ga(t)=x$ such that
\begin{equation}\label{minimizer}
u(t,x)=\fui(\ga(0))+\int_0^tL_0(\dga(\tau),u(\tau,\ga(\tau)))d\tau.
\end{equation}
Then, we have the following two cases:
\begin{enumerate}[(I)]
\item There exists a $t_0\in [0,t)$ such that
  $u(t_0,\ga(t_0)) = 0$ and $u(\tau,\ga(\tau))< 0$ for $\tau\in (t_0,t]$.
\item For every $\tau\in[0,t]$,  $u(\tau,\ga(\tau)) < 0$.
\end{enumerate}
For Case (I)
\begin{align*}
u(x, t)&=u(t_0,\ga(t_0)) + \int_{t_0}^t L_0( \dga(\tau), u(\tau,\ga(\tau)))\ d\tau
 \ge\int_{t_0}^t L(\dga(\tau),0)+c[0]\ d\tau\\
&\ge \begin{cases}-(C(0)+c[0])_- & t\le t_0+1\\-C_1 & t>t_0+1 \end{cases}.
  \end{align*}
For Case (II), 
 \begin{align*}
u(t,x)&\ge\fui(\ga(0))+\int_0^tL(\dga(\tau), 0)+c[0]\ d\tau
\ge\fui(\ga(0))+h_t(\ga(0),\ga(t))+c[0]t\\&\ge\min\fui-C_1
 \end{align*}  

We now show that $u(t,x)$ is uniformly bounded from above.
Assume  there is $(t,x)\in (1,+\infty)\times M$ with $u(t,x) > 0$,
otherwise $u$ is bounded from above.
There exists $\ga\in\cWD([0,t])$ with $\ga(t)=x$ satisfying \eqref{minimizer}.
Let $\bar{\ga}\in\cWD([0,t])$ with $\bar\ga(0)=\ga(0)$, $\bar\ga(t)=x$
such that
\[\int_0^tL(\dot{\bar\ga}(\tau), 0)\ d\tau=h_t(\ga(0),x)\]
Then, we have the following two cases:
\begin{enumerate}[(I)]
\item There exists a $t_0\in [0,t)$ such that
  $u(t_0,\bar\ga(t_0)) = 0$ and $u(\tau,\bar\ga(\tau))>0$ for $\tau\in (t_0,t]$.
\item For every $\tau\in[0,t]$,  $u(\tau,\bar\ga(\tau)) > 0$.
\end{enumerate}
 For Case (I), we have
\begin{align*}
  u(t,x) &\le u(t_0,\bar\ga(t_0))
     +\int_{t_0}^t L_0(\dot{\bar\ga}(\tau), u(\tau,\bar\ga(\tau)))\ d\tau
       \le\int_{t_0}^t L(\dot{\bar\ga}(\tau),0)+c[0]\ d\tau\\
       &=h_{t-t_0}(\bar\ga(t_0), x) +c[0](t-t_0).
\end{align*}

Since $\bar\ga$ is a minimizer for $L(\cdot,0)$
\[h_{t- t_0}(\bar\ga(t_0), x)=h_t(\ga(0),x)-h_{t_0}(\ga(0), \bar\ga(t_0)),\]
thus
\[u(t,x)\le h_{t-t_0}(\bar\ga(t_0),x) +c[0](t-t_0)\le
  \begin{cases}
 C_\frac12    &t_0\le\frac 12\\
 C_1+ C_\frac 12&t_0>\frac 12
\end{cases}.\]

For Case (II), we have
\begin{align*}
  u(t,x) 
       &\le\fui(\ga(0))+ \int_0^t L(\dot{\bar\ga}(\tau),0)+c[0]\ d\tau=
         \fui(\ga(0))+h_t(\ga(0), x) +c[0]t\\
&\le \max\fui+C_1 
\end{align*}
\end{proof}
\begin{lemma}[Equi-Lipschitz]\label{equiLip}
  Suppose that $\fui\in\Lip(M,d)$. Then there is $K>0$ such that
  $|T^*_t[\fui](x)-T^*_t[\fui](y)|\le K d(x,y)$ for all $x,y\in M$,
  $t\ge 0$. 
  \end{lemma}
  \begin{proof}
    Let  $u(t,x)=T^*_t[\fui](x)=\A_\fui[u](t,x)+c_0t$, by
    Lemma \ref{unif-bound} $\|u\|_\infty<\infty$.
    By Proposition \ref{non-exp}, \eqref{eq:s,t} and \eqref{eq:t,s}, for $t,r\ge 0$ we have
    \begin{equation}
      \label{eq:lip-t}
      \|T^*_{t+r}[\fui]-T^*_t[\fui]\|_\infty=\|T^*_t[T^*_r[\fui]]-T^*_t[\fui]\|_\infty\le
      \|T^*_r[\fui]-\fui]\|_\infty\le K_1r
          \end{equation}
where $K_1:=\max (|A(0)|,|C(\Lip(\fui))|)+|c_0|+\la \|u\|_\infty$.

Let $x,y\in M$, $t\ge 0$, by \eqref{eq:lip-t} and \eqref{t+d},  we have
\[T^*_t[\fui](y)\le T^*_{t+d(x,y)}[\fui](y)+K_1 d(x,y)\le T^*_t[\fui](x)+(K_1+K_2)d(x,y)\]
where $K_2:=A(1)+c_0+\la\|u\|_\infty$.
  \end{proof}
  \begin{lemma}\label{main}
    Suppose $\fui\in\Lip(M,d)$ and let $\ovu=\limsup\limits_{t\to\infty}T^*_t[\fui]$,
    $\unu=\liminf\limits_{t\to\infty}T^*_t[\fui]$.
    Then
    \begin{enumerate}[(a)]
    \item $\ovu,\unu\in\Lip(M,d)$
    \item  $u^\infty=\lim\limits_{t\to\infty}T^*_t[\ovu]$,
      $u_\infty=\lim\limits_{t\to\infty}T^*_t[\unu]$ exist and  $T^*_t[u^\infty]=u^\infty$,
      $T^*_t[u_\infty]=u_\infty$ for any $t\ge 0$, so that
      $u^\infty, u_\infty$ are  viscosity solutions of \eqref{eq:HJ}.
    \item Under assumption \ref{Bar}, $T^*_t[\fui]$ converges
      uniformly, as $t\to\infty$, to a  viscosity solution of \eqref{eq:HJ}.
      \end{enumerate}
  \end{lemma}
  \begin{proof} {\it (a)} 
    Let $K>0$ be given by Lemma \ref{equiLip}.
    Let $v(s,\cdot)=\sup\limits_{t\ge s}T^*_t[\fui]$. For $x,y\in M$
    \[T^*_t[\fui](x)-Kd(x,y)\le T^*_t[\fui](y)\le T^*_t[\fui](x)+Kd(x,y),\]
    taking $\sup$ over $t\ge s:$ \, $v(s,x)-Kd(x,y)\le v(s,y)\le
    v(s,x)+Kd(x,y),$ \newline
        taking $\inf$ over $s\ge 0:$ \, $\ovu(x)-Kd(x,y)\le\ovu(y)\le\ovu(x)+Kd(x,y).$
        \newline Similarly, $|\unu(x)-\unu(y)|\le Kd(x,y)$.
        
  {\it (b)}  Let $\ep>0$. For each $x\in M$ there is $s(x)>0$ such that for $t\ge s(x)$
     \[T^*_t[\fui](x)<\ovu(x)+\ep,\]
   and if $d(x,y)<\ep/K$
     \[T^*_t[\fui](y)<T^*_t[\fui](x)+\ep<\ovu(x)+2\ep<\ovu(y)+3\ep.\]
Choose $x_1,\ldots,x_k$ such that $M=\bigcup_{i=1}^kB(x_i,\ep/K)$.
Let $s=\max\{s(x_i):i=1,\ldots k\}$. Let $t\ge s$ and $y\in M$, choose
$i=1,\ldots k$ such that $d(x_i,y)<\ep/K$ then
\[T^*_t[\fui](y)<\ovu(y)+3\ep.\]
By Propositions \ref{non-exp} and \ref{monotony}, for any $r\ge 0$
\[T^*_{r+t}[\fui]=T^*_r[T^*_t[\fui]]\le T^*_r[\ovu+3\ep]\le T^*_r[\ovu]+3\ep,\]
taking $\limsup$ as $t\to\infty$: $\ovu\le T^*_r[\ovu]+3\ep.$
Since $\ep>0$ is arbitrary, $\ovu\le T^*_r[\ovu]$, and so by Proposition \ref{monotony},
the family $(T^*_t[\ovu])_t$ is nondecreasing. Similarly, the family  $(T^*_t[\unu])_t$ is
nonincreasing. Thus, by Lemmas \ref{unif-bound} and \ref{equiLip},
the uniform limits $u^\infty=\lim\limits_{t\to\infty}T^*_t[\ovu]$
and $u_\infty=\lim\limits_{t\to\infty}T^*_t[\unu]$ exist, $u_\infty, u^\infty\in\Lip(M,d)$
and $u^\infty\ge\ovu$, $u_\infty\le\unu$.
By Proposition \ref{non-exp}, $T^*_t[u^\infty]=u^\infty$, $T^*_t[u_\infty]=u_\infty$
for any $t\ge 0$, and then $u^\infty, u_\infty$ are viscosity solutions of \eqref{eq:HJ}.

{\it (c)} We have $u_\infty\le\unu\le\ovu\le u^\infty$, and by
assumption \ref{Bar}, $u_\infty=u^\infty$. Thus $\unu=\ovu$ and by
Lemma \ref{equiLip} the uniform limit $u=\lim\limits_{t\to\infty}T^*_t[\fui]$
exists and by Proposition \ref{non-exp},
 $T^*_t[u]=u$ for any $t\ge 0$.
\end{proof}

\end{document}